\documentclass[10pt]{amsart}
\usepackage{mythmsub}
\usepackage{tikz-cd}
\usepackage[shortlabels]{enumitem}
\numberwithin{equation}{thm}
\newcommand{\eff}[0]{{e\!f\!f}}
\newcommand{\dmeff}[0]{{\rm DM}^\eff({\bf C})}
\newcommand{\dmeffl}[0]{{\rm DM}^\eff({\bf C},{\bf Q})}
\begin{document}
\title{Intermediate Jacobians and the slice filtration}
\author{Doosung Park}
\date{\today}
\address{Institut f\"ur Mathematik, Universit\"at Z\"urich, Winterthurerstr. 190, 8057 Z\"urich, Switzerland}
\email{doosung,park@math.uzh.ch}
\subjclass[2010]{14C15}
\keywords{intermediate Jacobians, motives}
\begin{abstract}
For every $n$-dimensional smooth projective variety $X$ over ${\bf C}$, the motive $M(X)$ is expected to admit a Chow-K\"unneth decomposition $M_0(X)\oplus\cdots \oplus M_{2n}(X)$.
Inspired by the slice filtration of $M(X)$ we propose the definitions of $M_2(X)$ and $M_{2n-2}(X)$.
In our construction we use intermediate Jacobians.
\end{abstract}
\maketitle
\section{Introduction}
Grothendieck entertained the notion of mixed motives as a universal cohomology theory, which would encompass the Betti, $\ell$-adic, de Rham, and crystalline cohomology theories. While the category of mixed motives ${\rm MM}$ remains elusive, there are several constructions of a triangulated category of mixed motives ${\rm D}({\rm MM})$ suggested independently by Hanamura, Levine, and Voevodsky.
In this paper, we use Voevodsky's construction of effective motives ${\rm DM}^{e\! f\! f}(k,\Lambda)$ \cite{MVW}, where $k$ is a field and $\Lambda$ is a coefficient ring.\vspace{0.1in}

The problem of constructing the category of mixed motives $\mathrm{MM}$ is closely related to Grothendieck's standard conjectures on algebraic cycles \cite{MR0268189}.  One of his conjectures, the K\"unneth-type standard conjecture, can be reformulated in terms of $\dmeffl$ when the base field is ${\bf C}$ as follows:
For every $n$-dimensional smooth projective variety $X$ over ${\bf C}$ the motive $M(X)$ admits a decomposition
\[
M(X)
\cong
M_0(X)\oplus M_1(X)\oplus \cdots \oplus M_{2n}(X)
\]
such that for every integer $d\in [0,2n]$ the composite morphism
\[
M(X)\to M_d(X)\to M(X)
\]
corresponds to the obvious composite morphism
\[
\bigoplus_{i=0}^{2n} H_{\rm Betti}^i(X,{\bf Q})
\to
H_{\rm Betti}^d(X,{\bf Q})
\to
\bigoplus_{i=0}^{2n} H_{\rm Betti}^i(X,{\bf Q}).
\]
The decomposition is called a \emph{Chow-K\"unneth decomposition} of $M(X)$.
The constructions of $M_0(X)$ and $M_{2n}(X)$ are straightforward, and Murre constructed $M_1(X)$ and $M_{2n-1}(X)$ in \cite{MR1061525}.
We also refer to \cite{MR1265529} for a survey of the theory of classical motives.\vspace{0.1in}
The construction of $M_d(X)$ for general $d$ is unknown.

In this paper we propose the definitions of $M_2(X)$ and $M_{2n-2}(X)$.
Our method is inspired by the slice filtration \cite{MR2249535} of $M(X)$:
\begin{align*}
\underline{\rm Hom}({\bf L}^{n},M(X))\otimes {\bf L}^n \rightarrow & \underline{\rm Hom}({\bf L}^{n-1},M(X))\otimes {\bf L}^{n-1}
\\
\rightarrow & \cdots \rightarrow \underline{\rm Hom}({\bf 1},M(X))=M(X).
\end{align*}
Here, $\underline{\rm Hom}$ denotes the internal Hom of $\dmeffl$, and
\[
{\bf 1}:=M({\rm Spec}\,k),\quad {\bf L}:={\bf 1}(1)[2].
\]
For every integer $i\geq 0$ a Chow-K\"unneth decomposition would naturally induce a decomposition
\[
\underline{\rm Hom}({\bf L}^i,M(X))
\cong
\underline{\rm Hom}({\bf L}^i,M_0(X))
\oplus
\cdots
\oplus
\underline{\rm Hom}({\bf L}^i,M_{2n}(X)).
\]
We expect that there is an isomorphism
\[
\underline{\rm Hom}({\bf L}^i,M_{n+i}(X))\cong M_{n-i}(X).
\]
Hence to propose the definition of $M_{n-i}(X)$ we may need to identify which components of the motive $\underline{\rm Hom}({\bf L}^i,M(X))$ correspond to the motives $\underline{\rm Hom}({\bf L}^i,M_j(X))$ for $j\in [0,n+i)\cup (n+i,2n]$.
This method is used to produce $M_1(X)$.
More precisely, there is a decomposition
\[
\underline{\rm Hom}({\bf L}^{n-1},M(X))
\cong
{\rm NS}_{\bf Q}(X)\oplus {\rm Pic}_{\bf Q}^0(X)\oplus {\bf L}.
\]
The summand ${\rm Pic}_{\bf Q}^0(X)$ was identifies as $M_1(X)$.
We generalize this decomposition to the case $i=n-2$, which is achieved by the following two theorems.

\begin{thm}\label{0.1}
Let $X$ be an $n$-dimensional smooth projective variety over $\mathbf{C}$, and let $d\in [1,n]$ be an integer.
Then
\begin{enumerate}[{\rm (1)}]
\item
${\rm NS}_{hom,{\bf Q}}^d(X)\oplus {\rm Griff}_{\bf Q}^d(X)$ is a direct summand of $\underline{\rm Hom}(\mathbf{L}^{n-d},M(X))$ in $\dmeffl$,
\item
$J_{a,{\bf Q}}^d(X)$ is a direct summand of $\underline{\rm Hom}(\mathbf{L}^{n-d},M(X))$ in $\dmeffl$.
\end{enumerate}
\end{thm}

See {\rm Definition \ref{1.8}} and {\rm \ref{2.4}}(4) for the definitions of ${\rm NS}_{hom,{\bf Q}}^d(X)$, ${\rm Griff}_{\bf Q}^d(X)$, and $J_{a,{\bf Q}}^d(X)$.

\begin{thm}
\label{0.3}
Let $X$ be an $n$-dimensional smooth projective variety over $\mathbf{C}$ with $n\geq 2$. Then for some motive $M_2(X)$ in $\dmeffl$ there is a decomposition
\begin{align*}
& \underline{\rm Hom}({\bf L}^{n-2},M(X))
\\
\cong\, &
{\rm NS}_{hom,{\bf Q}}^2(X)\oplus {\rm Griff}_{\bf Q}^2(X)\oplus J_{a,{\bf Q}}^2(X)\oplus M_2(X)\oplus ( {\rm Pic}_{\bf Q}^0(X)\otimes {\bf L})\oplus {\bf L}^2.
\end{align*}
\end{thm}
This is our definition of $M_2(X)$.
Our definition of $M_{2n-2}(X)$ is
\[
M_{2n-2}(X):=M_2(X)\otimes {\bf L}^{n-2}.
\]
What we do not know is whether $M_2(X)$ and $M_{2n-2}(X)$ are direct summands of $M(X)$ or not.
We do not expect that this question is easily reachable since a positive answer would lead to the proof of the K\"unneth type standard conjecture for smooth projective threefolds over ${\bf C}$.
What we hope in the future is to explore the properties of the motives $M_2(X)$ and $M_{2n-2}(X)$ that may be useful for understanding the motivic structures inherited in $X$.

\begin{none}\label{0.9}
 {\it Organization of the paper.} In Section 2, we prove Theorem \ref{0.1}(1) by constructing a morphism $\underline{\rm Hom}(\mathbf{L}^{n-d},M(X))\rightarrow {\rm NS}_{alg,{\bf Q}}^d(X)$ and its section. In Section 3, we prove Theorem \ref{0.1}(2) by constructing a morphism $\underline{\rm Hom}(\mathbf{L}^{n-d},M(X))\rightarrow J_{a,{\bf Q}}^d(X)$ and its section. In Section 4, we prove Theorem \ref{0.3} by constructing the other pieces and using \cite[Theorem 7.3.10]{MR2187153}. In Section 5, we discuss some conjectures.
\end{none}
\begin{none}\label{0.10}
 {\it Conventions and notations.}
 \begin{enumerate}[(1)]
  \item For every complex analytic variety or a scheme $T$ over ${\bf C}$, let ${\rm cl}(T)$ denote the set of closed points of $T$.
  \item Let $Sm/{\bf C}$ denote the category of smooth schemes over ${\bf C}$.
  \item For any ${\bf Q}$-vector space $V$, we also view it as an object in $\dmeffl$.
  \item For every abelian variety $A$, we also view it as an object in $\dmeff$ or $\dmeffl$ by using \cite{MR2102056}.
 \end{enumerate}
\end{none}
\section{Proof of Theorem \ref{0.1}(1)}
\begin{df}\label{1.8}
For all smooth scheme $X$ over ${\bf C}$ and integer $d\geq 0$ we set
\begin{gather*}
CH_{alg}^d(X):=\{Z\in CH^d(X):Z\sim_{alg}0\},
\\
CH_{hom}^d(X):=\{Z\in CH^d(X):Z\sim_{hom}0\},
\\
{\rm NS}_{alg}^d(X):=CH^d(X)/CH_{alg}^d(X),
\;
{\rm NS}_{hom}^d(X):=CH^d(X)/CH_{hom}^d(X),
\\
{\rm Griff}^d(X):=CH_{hom}^d(X)/CH_{alg}^d(X).
\end{gather*}
Here, $\sim_{alg}$ (resp.\ $\sim_{hom}$) denotes the algebraic equivalence relation (resp.\ homological equivalence relation for the Betti cohomology).
Let
\[
CH_{alg,{\bf Q}}^d(X),
\;
CH_{hom,{\bf Q}}^d(X),
\;
{\rm NS}_{alg,{\bf Q}}^d(X),
\;
{\rm NS}_{hom,{\bf Q}}^d(X),
\;
{\rm Griff}_{\bf Q}^d(X)
\]
denote the corresponding ones defined for ${\bf Q}$-coefficient.
\end{df}

\begin{df}\label{1.1}
For all smooth schemes $X$ and $Y$ over ${\bf C}$ and integer $d\geq 0$ we set
\[CH_{X}^d(Y):=CH^d(Y\times X).\]
When $Y$ is connected, we set
\begin{gather*}
CH_{alg,X}^d(Y):=\{Z\in CH^d(Y\times X):i_y^*Z\sim_{alg} 0\},
\\
CH_{hom,X}^d(Y):=\{Z\in CH^d(Y\times X):i_y^*Z\sim_{hom} 0\},
\\
{\rm NS}_{alg,X}^d(Y):=CH_X^d(Y)/CH_{alg,X}^d(Y),
\;
{\rm NS}_{hom,X}^d(Y):=CH_X^d(Y)/CH_{hom,X}^d(Y).
\end{gather*}
Here, $y$ is any closed point of $Y$, and $i_y$ is the obviously closed immersion $y\times X\rightarrow Y\times X$.
We note that the above definitions are independent of the choice $y$ since $i_y^*Z$ and $i_{y'}^*Z$ are algebraically equivalent for arbitrary two closed points $y$ and $y'$ of $Y$.\vspace{0.1in}

When $Y$ is not necessarily connected and has the connected components $\{Y_i\}_{i\in I}$, we set
\[CH_{alg,X}^d(Y):=\bigoplus_{i\in I}CH_{alg,X}^d(Y_i).\]
We define $CH_{hom,X}^d(Y)$, ${\rm NS}_{alg,X}^d(Y)$, and ${\rm NS}_{hom,X}^d(Y)$ similarly.
We consider
\[
CH_X^d,\;
CH_{alg,X}^d,\;
CH_{hom,X}^d,\;
{\rm NS}_{alg,X}^d,\;
{\rm NS}_{hom,X}^d
\]
as presheaves with transfers on $Sm/{\bf C}$.
Let
\[CH_{X,{\bf Q}}^d,\;\;CH_{alg,X,{\bf Q}}^d,\;\;CH_{hom,X,{\bf Q}}^d,\;\;{\rm NS}_{alg,X,{\bf Q}}^d,\;\;{\rm NS}_{hom,X,{\bf Q}}^d\]
denote the corresponding ones defined for ${\bf Q}$-coefficient.
\end{df}
\begin{prop}\label{1.2}
Let $X$ and $Y$ be smooth schemes over ${\bf C}$.
For every closed point $i_y\colon y={\rm Spec}\,{\bf C}\to Y$ and integer $d\geq 0$ the homomorphisms
\[i_y^*\colon {\rm NS}_{alg,X}^d(Y)\rightarrow {\rm NS}_{alg,X}^d({\rm Spec}\,{\bf C}),\;
i_y^*\colon {\rm NS}_{hom,X}^d(Y)\rightarrow {\rm NS}_{hom,X}^d({\rm Spec}\,{\bf C})\]
are isomorphisms.
\end{prop}
\begin{proof}
The homomorphisms are injective since the kernels of the homomorphisms
\[CH^d(Y\times X)\rightarrow {\rm NS}_{alg}^d(X),\; CH^d(Y\times X)\rightarrow {\rm NS}_{hom}^d(X)\]
are given by
\[\{Z\in CH^d(Y\times X):i_y^*Z\in CH_{alg}^d(X)\}, \{Z\in CH^d(Y\times X):i_y^*Z\in CH_{hom}^d(X)\}.\]
The homomorphisms are surjective since $i_y^*\colon CH^d(Y\times X)\rightarrow CH^d(X)$ is surjective.
\end{proof}

\begin{cor}\label{1.3}
The presheaves ${\rm NS}_{alg,X}^d$ and ${\rm NS}_{hom,X}^d$ on $Sm/X$ are constant Nisnevich sheaves with transfers associated to ${\rm NS}_{alg}^d(X)$ and ${\rm NS}_{hom}^d(X)$ respectively.
\end{cor}
\begin{proof}
Let $Y$ be an integral smooth variety over ${\bf C}$ with the structure morphism $p\colon Y\rightarrow {\rm Spec}\,{\bf C}$.
The homomorphism
\[
p^*
\colon
{\rm NS}_{alg,X}^d({\rm Spec}\,{\bf C})
\to
{\rm NS}_{alg,X}^d(Y)
\]
is a section of the homomorphism
\[
i_y^*\colon {\rm NS}_{alg,X}^d(Y)\rightarrow {\rm NS}_{alg,X}^d({\rm Spec}\,{\bf C}),
\]
where $y$ is a closed point of $Y$ and $i_y\colon y\to Y$ is the closed immersion.
Since $i_y^*$ is an isomorphism owing to Proposition \ref{1.2}, $p^*$ is an isomorphism too.
It follows that ${\rm NS}_{alg,X}^d$ is a constant Zariski sheaf associated with
\[
{\rm NS}_{alg,X}^d({\rm Spec}\,{\bf C})= {\rm NS}_{alg}^d(X).
\]
The proof for ${\rm NS}_{hom,X}^d$ is similar.
\end{proof}
\begin{none}\label{1.4}
Note that Corollary \ref{1.3} also holds for ${\bf Q}$-coefficient. Thus from now, we use the notations
\[
{\rm NS}_{alg,{\bf Q}}^d(X),\;\;{\rm NS}_{hom,{\bf Q}}^d(X)
\]
instead of ${\rm NS}_{alg,X,{\bf Q}}^d$ and ${\rm NS}_{hom,X,{\bf Q}}^d$ following the convention in \ref{0.10}(3).
\end{none}
\begin{df}\label{1.5}
For every $i\in {\bf Z}$, let
\[
h_i\colon \dmeff\rightarrow {\rm Sh}^{tr}(Sm/{\bf C})
\]
denote the homology functor for the homotopy $t$-structure defined in \cite[Definition 3.1]{MR2735752}. Here, ${\rm Sh}^{tr}(Sm/{\bf C})$ denotes the category of sheaves with transfers on $Sm/{\bf C}$ with ${\bf Q}$-coefficient.
Let $\tau_{\geq i}$ denote the homological truncation functor associated with the homotopy $t$-structure.
\end{df}
\begin{none}\label{1.6}
Let $X$ be an $n$-dimensional smooth projective variety over ${\bf C}$.
As in \cite[Section A.3]{MR2473327} we have
\[
h_i(\underline{\rm Hom}({\bf L}^{n-d},M(X)))
\cong
\left\{
\begin{array}{ll}
0&\text{if }i<0,
\\
CH_{X}^d&\text{if }i=0,
\end{array}
\right.
\]
where $\underline{\rm Hom}$ denotes the internal Hom in $\dmeff$.
From this we can make morphisms
\begin{equation}
\label{1.6.4}
\begin{split}
&\underline{\rm Hom}({\bf L}^{n-d},M(X))
\xrightarrow{\sim}
\tau_{\geq 0}\underline{\rm Hom}({\bf L}^{n-d},M(X)))
\\
\to &
h_0(\underline{\rm Hom}({\bf L}^{n-d},M(X)))
\xrightarrow{\sim}
CH_{X}^d.
\end{split}
\end{equation}
There is a morphism
\begin{equation}
\label{1.6.3}
CH_{X}^d \rightarrow {\rm NS}_{alg,X}^d={\rm NS}_{alg}^d(X)
\end{equation}
obtained by taking the quotient of $CH^d(Y\times X)$ for $Y\in Sm/{\bf C}$.\vspace{0.1in}

Combine \eqref{1.6.4} and \eqref{1.6.3} to have a morphism
\begin{equation}\label{1.6.2}
\underline{\rm Hom}({\bf L}^{n-d},M(X))\rightarrow {\rm NS}_{alg}^d(X).
\end{equation}
The ${\bf Q}$-coefficient version is
\begin{equation}\label{1.6.1}
\underline{\rm Hom}({\bf L}^{n-d},M(X))\rightarrow {\rm NS}_{alg,{\bf Q}}^d(X).
\end{equation}
\end{none}
\begin{prop}\label{1.7}
Let $X$ be an $n$-dimensional smooth projective variety over ${\bf C}$.
Then the morphism $\eqref{1.6.1}$ has a section in $\dmeffl$.
\end{prop}
\begin{proof}
Since ${\rm NS}_{alg,{\bf Q}}^d(X)$ is a ${\bf Q}$-vector space, it has a basis $\{a_i\}_{i\in I}$ for some set $I$.
Then ${\rm NS}_{alg,{\bf Q}}^d(X)$ is isomorphic to $\bigoplus_{i\in I}{\bf Q}$.
In $\dmeffl$, there is an isomorphism.
\[
{\rm NS}_{alg,{\bf Q}}^d(X)\cong\bigoplus_{i\in I} {\bf 1}.
\]
  
There are bijections
\begin{equation}\label{1.7.1}
\begin{split}
&{\rm Hom}_{\dmeffl}({\rm NS}_{alg,{\bf Q}}^d(X),\underline{\rm Hom}({\bf L}^{n-d},M(X)))
\\
\cong \,& {\rm Hom}_{\dmeffl}({\rm NS}_{alg,{\bf Q}}^d(X)\otimes {\bf L}^{n-d},M(X))
\\
\cong \,& I\times {\rm Hom}_{\dmeffl}( {\bf L}^{n-d},M(X))
\\
\cong \,& I\times CH_{{\bf Q}}^d(X)
\\
\cong \,& {\rm Hom}_{\rm Set}(I,CH_{{\bf Q}}^d(X)),
\end{split}\end{equation}
where $\rm Set$ denotes the category of sets.
Choose a set
\[
\{b_i\in CH_{\bf Q}^d(X)\}_{i\in I}
\]
such that the image of $b_i$ in ${\rm NS}_{alg,{\bf Q}}^d(X)$ is $a_i$. Via (\ref{1.7.1}) the function $I\rightarrow CH^d(X)$ sending $i$ to $b_i$ corresponds to a section of (\ref{1.6.1}).
\end{proof}
\begin{none}\label{1.9}
The quotient homomorphism ${\rm NS}_{alg,{\bf Q}}^d(X)\rightarrow {\rm NS}_{hom,{\bf Q}}^d(X)$ has a section since they are ${\bf Q}$-vector spaces.
Thus there is a decomposition
\[{\rm NS}_{alg,{\bf Q}}^d(X)\cong {\rm NS}_{hom,{\bf Q}}^d(X)\oplus {\rm Griff}_{\bf Q}^d(X),\]
and then the proof of Theorem \ref{0.1}(1) is completed by Proposition \ref{1.7}.
\end{none}
\section{Proof of Theorem \ref{0.1}(2)}
\begin{lemma}\label{2.5}
Let $X$ and $Y$ be schemes of finite type over an algebraically closed field $k$.
Assume that $X$ is integral and each connected component of $Y$ is integral.
If $X$ is quasi-projective over $k$, then for every function $f\colon {\rm cl}(Y)\rightarrow {\rm cl}(X)$ there are at most one morphism $Y\rightarrow X$ of schemes over $k$ naturally inducing $f$.
\end{lemma}
\begin{proof}
The question is Zariski local on $Y$, so we may assume $Y$ is integral and affine.
The statement follows from the classical fact that quasi-projective varieties over $k$ form a full subcategory of schemes over $k$.
\end{proof}
\begin{none}\label{2.4}
Here, we review several facts about intermediate Jacobians and Abel-Jacobi maps.
Let $X$ be an $n$-dimensional smooth projective variety over ${\bf C}$, and let $d\in [1,n]$ be an integer.
\begin{enumerate}[(1)]
\item
For every $x\in {\rm cl}(X)$ there is a morphism
\[
{\rm Alb}_{X,x}\colon X\rightarrow {\rm Alb}(X)
\]
mapping $x$ to $0$, which is called the \emph{Albanese map}.
\item
There is a complex torus $J^d(X)$, which is called an intermediate Jacobian.
We refer to \cite[Definition 12.2]{MR2451566} for the definition.
Note that $J^d(X)$ is not necessarily an abelian variety.
\item
There is a homomorphism
\[AJ_X^d\colon CH_{hom}^d(X)\rightarrow {\rm cl}(J^d(X)),\]
which is called the \emph{Abel-Jacobi map}.
We refer to \cite[p.\ 294]{MR2451566} for the definition.
\item
There is a sub-torus $J_a^d(X)$ of $J^d(X)$, which is an abelian variety.
Moreover, there is a commutative diagram
\[
\begin{tikzcd}
CH_{alg}^d(X)\arrow[d]\arrow[r,"AJ_X^d"]&{\rm cl}(J_a^d(X))\arrow[d]\\
CH_{hom}^d(X)\arrow[r,"AJ_X^d"]&{\rm cl}(J^d(X))
\end{tikzcd}
\]
The vertical morphisms are the obvious inclusions of abelian groups, and the upper horizontal morphism is surjective.
When $d=n$, we have
\[
J^n(X)=J_a^n(X)={\rm Alb}(X).
\]
We refer to \cite{MR3141813} for details.
\item Let $Y$ be a smooth scheme over ${\bf C}$.
By \cite[Section 4]{MR0424832}, there is a homomorphism
\[
AJ_{X,Y}^d\colon CH_{alg,X}^d(Y)\rightarrow {\rm Hom}_{{\rm Sch}_{\bf C}}(Y,J_a^d(X))
\]
of abelian groups, where, ${\rm Sch}_{\bf C}$ denotes the category of ${\bf C}$-schemes.
The morphism $AJ_{X,Y}^d(Z)$ sends $Z\in CH_{alg,X}^d(Y)$ to a morphism
\[
Y\rightarrow J_a^d(X)
\]
sending $y\in {\rm cl}(Y)$ to $AJ_X^d(i_y^*Z)$, where $i_y\colon y\times X\rightarrow Y\times X$ denotes the obvious closed immersion.
Observe that by Lemma \ref{2.5} $AJ_{X,Y}^d$ is uniquely determined by the above description.
\item Let $Y$ be an $m$-dimensional smooth projective variety over ${\bf C}$.
If $Z\in CH_X^d(Y)$ is a cycle, then there is a homomorphism
\[
\psi_Z\colon {\rm Alb}(Y)\rightarrow J_a^d(X)
\]
of abelian varieties naturally induced by the morphism of the Hodge structures
\[H^{2m-1}(Y,{\bf Z})\rightarrow H^{2d-1}(X,{\bf Z})\]
naturally induced by $Z$.
We refer to \cite[Theorem 12.17]{MR2451566} for details.
By \cite[Sections 3, 4]{MR0424832}, for all $y\in {\rm cl}(Y)$ and $Z\in CH_{X}^d(Y)$ there is a commutative diagram
\begin{equation}
\label{2.4.1}\begin{tikzcd}
Y\arrow[d,"{\rm Alb}_{Y,y}"']\arrow[rd,"AJ_{X,Y}^d(Z')"]\\
{\rm Alb}(Y)\arrow[r,"\psi_Z"']&J_a^d(X)
\end{tikzcd}
\end{equation}
of schemes.
Here, $Z':=Z-y\times i_y^*Z$, and $i_y\colon y\times X\rightarrow Y\times X$ denotes the obvious closed immersion.
 \end{enumerate}
\end{none}
\begin{prop}\label{2.2}
Let $X$ be an $n$-dimensional smooth projective variety over ${\bf C}$.
Then
\[
AJ_{X,-}^d\colon CH_{alg,X}^d\rightarrow J_a^d(X)
\]
is a morphism of presheaves with transfers on $Sm/{\bf C}$ for every integer $1\leq d\leq n$.
\end{prop}
\begin{proof}
Let $Y$ and $Y'$ be noetherian smooth schemes over ${\bf C}$.
We need to show that for every finite correspondence $V$ from $Y'$ to $Y$ the diagram
\begin{equation}\label{2.2.1}\begin{tikzcd}
CH_{alg,X}(Y)\arrow[r,"AJ_{X,Y}^d"]\arrow[d,"\alpha"']&
{\rm Hom}_{{\rm Sch}_{\bf C}}(Y,J_a^d(X))\arrow[d,"\beta"]
\\
CH_{alg,X}(Y')\arrow[r,"AJ_{X,Y'}^d"]&
{\rm Hom}_{{\rm Sch}_{\bf C}}(Y',J_a^d(X))
\end{tikzcd}
\end{equation}
of abelian groups commutes, where ${\rm Sch}_{\bf C}$ denotes the category of ${\bf C}$-schemes, and $\alpha$ and $\beta$ denote the homomorphisms induced by $V$.
To show this we may assume that $Y$ and $Y'$ are connected and $V$ is an elementary finite correspondence.\vspace{0.1in}

We will review the definition of $\beta$ given in \cite[Lemme 3.1.2]{MR2102056}.
Let $f\colon Y'\rightarrow J_a^d(X)$ be a morphism of schemes. If $V$ has degree $r$, then $V$ naturally induces a morphism
\begin{equation}
\label{2.2.2}
Y'\rightarrow Y^{(r)}.
\end{equation}
There are morphisms
\begin{equation}
\label{2.2.3}
Y^{(r)}\stackrel{f^{(r)}}\rightarrow (J_a^d(X))^{(r)}\stackrel{\rm sum}\rightarrow J_a^d(X)
\end{equation}
of schemes, where $Y^{(r)}$, $(J_a^d(X))^{(r)}$, and $f^{(r)}$ denote the symmetric powers.
The morphism $\beta(f)\colon Y'\to J_a^d(X)$ is obtained by composing \eqref{2.2.2} and \eqref{2.2.3}.\vspace{0.1in}

Let $Z\in CH_{alg,X}^d(Y)$ be an element, and let $y'\in {\rm cl}(Y')$ be a closed point.
Then via $V$, $y'$ corresponds to
\[
a_1y_1+\cdots +a_sy_s
\]
for some $a_1\ldots,a_s\in {\bf N}^+$ and $y_1,\ldots,y_s\in {\rm cl}(Y)$.
By the description in \ref{2.4}(5) $AJ_{X,Y}^d(Z)$ maps $y\in {\rm cl}(Y)$ to $AJ_{X}^d(i_y^*Z)$, where $i_y\colon y\times X\rightarrow Y\times X$ is the obvious closed immersion.
Using the above description of $\beta$ we see that $\beta(AJ_{X,Y}^d(Z))$ maps $y'$ to
\[
a_1AJ_X^d (i_{y_1}^*Z)+\cdots +a_s AJ_X^d(i_{y_s}^*Z).\vspace{0.05in}
\]

Since $i_{y'}^*(\alpha(Z))=a_1i_{y_1}^*Z+\cdots +a_si_{y_s}^*Z$, we see that $AJ_{X,Y'}^d(\alpha (Z))$ maps $y'$ to
\[
AJ_X^d(a_1i_{y_1}^*Z+\cdots +a_si_{y_s}^*Z)=a_1AJ_X^d (i_{y_1}^*Z)+\cdots +a_s AJ_X^d(i_{y_s}^*Z).
\]
Thus $\beta(AJ_{X,Y}(Z))$ and $AJ_{X',Y'}(\alpha(Z))$ maps $y'$ to the same closed point of $J_a^d(X)$.
Apply Lemma \ref{2.5} to deduce that (\ref{2.2.1}) commutes.
\end{proof}

\begin{none}\label{2.3}
By Proposition \ref{2.2}, we can regard
\[AJ_{X,-}^d\colon {\rm CH}_{alg,X}^d\rightarrow J_a^d(X)\]
as a morphism in $\dmeff$. In (\ref{1.6.2}), we have the morphism
\begin{equation}
\label{2.3.1}
\underline{\rm Hom}({\bf L}^{n-d},M(X))\rightarrow {\rm NS}_{alg}^d(X)
\end{equation}
in $\dmeff$.
Let $K$ denote its cocone.\vspace{0.1in}

By \ref{1.6}, there is an isomorphism $h_0(K)\cong{\rm CH}_{alg,X}^d$, and we have the vanishing $h_i(K)=0$ for every $i<0$.
Thus there is an isomorphism
\[
\tau_{\geq 0}(K)
\cong
h_0(K).
\]
Then we have morphisms
\[
K\rightarrow \tau_{\geq 0}(K)\stackrel{\sim}\rightarrow h_0(K)\cong {\rm CH}_{alg,X}^d\stackrel{AJ_{X,Y}^d}\longrightarrow J_{a}^d(X)
\]
in $\dmeff$.
By tensoring with ${\bf Q}$, we have a morphism
\[
\gamma\colon K_{\bf Q}\rightarrow J_{a,{\bf Q}}^d(X)
\]
in $\dmeffl$, where $K_{\bf Q}:=K\otimes {\bf Q}$.
Our next goal is to construct its section in $\dmeffl$.\vspace{0.1in}

In \cite[Section 2.3.3]{MR3141813}, it is shown that there is a smooth projective curve $C$ over {\bf C} (not necessarily connected) and an element $Z\in CH^d(C\times X)$ such that the homomorphism
\[
\psi_Z\colon {\rm Alb}(C)\rightarrow J_a^d(X)
\]
of abelian varieties in \ref{2.4}(6) is surjective.\vspace{0.1in}

The image of the cycle $Z':=Z-y\times i_y^*Z$ in ${\rm NS}_{alg}^d(X)$ is $0$, where $i_y\colon y\times X\rightarrow Y\times X$ is the obvious closed immersion.
Hence $Z'$ gives a morphism
\[
M(C)\to \underline{\rm Hom}({\bf L}^{n-d},M(X))
\]
such that it becomes $0$ after composing with \eqref{1.6.1}.
In other words, we have a morphism
\[
M(C)\to K_{\bf Q}.
\]
From (\ref{2.4.1}), we have a commutative diagram
\[\begin{tikzcd}
M(C)\arrow[d]\arrow[r]&K_{\bf Q}\arrow[d,"\gamma"]\\
{\rm Alb}_{\bf Q}(C)\arrow[r,"\psi_{Z,{\bf Q}}"]&J_{a,{\bf Q}}^d(X)
\end{tikzcd}\]
in $\dmeffl$, where the left vertical morphism is the Albanese map.\vspace{0.1in}

Up to isogeny the category of abelian varieties over ${\bf C}$ is semi-simple, so $\psi_{Z,{\bf Q}}$ has a section since $\psi_Z$ is surjective.
The composition $M(C)\rightarrow J_{a,{\bf Q}}^d(X)$ has a section since ${\rm Alb}_{\bf Q}(C)$ is a direct summand of $M(C)$ in $\dmeffl$.
Thus $\gamma$ has a section.\vspace{0.1in}

This completes the proof of Theorem \ref{0.1}(2) since $K_{\bf Q}$ is a direct summand of $\underline{\rm Hom}({\bf L}^{n-d},M(X))$ by Theorem \ref{0.1}(1).
\end{none}

\section{Proof of Theorem \ref{0.3}}
\begin{lemma}\label{3.1}
Let $M$ be an object of $\dmeffl$, and let $\alpha,\beta\colon M\rightarrow M$ be projectors.
For abbreviation, we set
\[
F={\rm im}\,\alpha,\quad G={\rm im}\,\beta.
\]
If ${\rm Hom}_{\dmeffl}(G,F)=0$, Then $F\oplus G$ is a direct summand of $M$.
\end{lemma}
\begin{proof}
Since ${\rm Hom}_{\dmeffl}(G,F)=0$, we have $\alpha\beta=0$.
Using this, we deduce
\begin{gather*}
\alpha(\beta-\beta\alpha)=\alpha\beta-\alpha\beta\alpha=0,
\quad
(\beta-\beta\alpha)\alpha=\beta\alpha-\beta\alpha^2=0,
\\
(\beta-\beta\alpha)^2=\beta^2-\beta^2\alpha-\beta\alpha\beta+\beta\alpha\beta\alpha=\beta-\beta\alpha.
\end{gather*}
Thus $\beta-\beta\alpha$ is a projector orthogonal to $\alpha$.
Since we have
\begin{gather*}
\beta(\beta-\beta\alpha)\beta=\beta^3-\beta^2\alpha\beta=\beta,
\\
(\beta-\beta\alpha)\beta(\beta-\beta\alpha)
=
\beta^3-\beta\alpha\beta^2-\beta^3\alpha+\beta\alpha\beta^2\alpha=\beta-\beta\alpha,
\end{gather*}
we deduce ${\rm im}\,\beta\cong{\rm im}(\beta-\beta\alpha)$.
Thus $\alpha+\beta-\beta\alpha$ is a projector whose image is isomorphic to $F\oplus G$.
\end{proof}
\begin{none}\label{3.2}
Let $X$ be an $n$-dimensional smooth projective variety over $\mathbf{C}$ with $n\geq 2$, and let $x$ be a closed point of $X$.
Observe that ${\bf 1}$ and ${\rm Alb}_{\bf Q}(X)$ are direct summands of $M(X)$.
The motives
\[
{\bf L}^n\cong \underline{\rm Hom}({\bf 1},{\bf L}^n),\quad {\bf L}^{n-1}\otimes {\rm Pic}_{\bf Q}^0(X)\cong \underline{\rm Hom}({\rm Alb}_{\bf Q}(X),{\bf L}^n)
\]
are direct summands of $\underline{\rm Hom}(M(X),{\bf L}^n)$, which is isomorphic to $M(X)$ by \cite[Theorem 16.24]{MVW}.
Thus using \cite[Theorem 16.25]{MVW}, we deduce that ${\bf L}^2$ and ${\bf L}\otimes {\rm Pic}_{\bf Q}^0(X)$ are direct summands of $\underline{\rm Hom}({\bf L}^{n-2},M(X))$.
By \ref{1.9} there is an isomorphism
\[
{\rm NS}_{hom,{\bf Q}}^2(X)\oplus {\rm Griff}_{\bf Q}^2(X)\cong {\rm NS}_{alg,{\bf Q}}^2(X).
\]
Thus to prove Theorem \ref{0.3}, by virtue of Lemma \ref{3.1} it suffices to show
\begin{gather*}
{\rm Hom}_{\dmeffl}({\bf L}^2,{\rm Pic}_{\bf Q}^0(X)\otimes {\bf L})=0,
\quad
{\rm Hom}_{\dmeffl}({\bf L}^2,J_{a,{\bf Q}}^2(X))=0,
\\
{\rm Hom}_{\dmeffl}({\bf L}^2,{\rm NS}_{alg,{\bf Q}}^2(X))=0,
\\
{\rm Hom}_{\dmeffl}({\rm Pic}_{\bf Q}^0(X)\otimes {\bf L},J_{a,{\bf Q}}^2(X))=0,
\\
{\rm Hom}_{\dmeffl}({\rm Pic}^0(X)\otimes {\bf L},{\rm NS}_{alg,{\bf Q}}^2(X))=0,
\\
{\rm Hom}_{\dmeffl}(J_{a,{\bf Q}}^2(X),{\rm NS}_{alg}^2(X))=0.
\end{gather*}
These follow from \cite[Theorem 7.3.10]{MR2187153} because of the following reasons.
\begin{enumerate}[(i)]
\item The motive ${\bf L}^2$ is isomorphic to $M_4(S_0)$ for some $S_0$.
\item The motive ${\rm Pic}_{\bf Q}^0(X)\otimes {\bf L}$ is isomorphic to $M_3(S_1)$ for some $S_1$.
\item The motive $J_{a,{\bf Q}}^2(X)$ is isomorphic to $M_1(S_2)$ for some $S_2$.
\item The motive ${\rm NS}_{alg,{\bf Q}}^2(X)$ is isomorphic to $M_0(S_3)$ for some $S_3$.
\end{enumerate}
Here, $S_0$, $S_1$, $S_2$, and $S_3$ are (not necessarily connected) smooth projective surfaces over ${\bf C}$. This completes the proof of Theorem \ref{0.3}.
\end{none}

\section{Conjectures}
\begin{df}\label{4.5}
Let $X$ be an $n$-dimensional smooth projective variety over ${\bf C}$, and let $d\in [1,n]$ be an integer. Consider the homomorphism
\[
AJ_{X,{\bf Q}}^d\colon CH_{hom,{\bf Q}}^d(X)\rightarrow {\rm cl}(J_a^d(X))\otimes_{\bf Z}{\bf Q}
\]
of ${\bf Q}$-vector spaces induced by $AJ_X^d$.
We set
\[
CH_{Jac,{\bf Q}}^d(X):={\rm ker}\,AJ_{X,{\bf Q}}^d.
\]
\end{df}
\begin{conj}\label{4.3}
Let $X$ be an $n$-dimensional smooth projective variety over ${\bf C}$ with $n\geq 2$.
Then
\[
CH_{Jac,{\bf Q}}^2(X)\subset CH_{alg,{\bf Q}}^2(X).
\]
\end{conj}
\begin{none}\label{4.4}
Let us show Conjecture \ref{4.3} using some other conjectures.
We need to show that any element in the kernel of
\[
AJ_{X,{\bf Q}}^2\colon CH_{hom,{\bf Q}}^2(X)\rightarrow {\rm cl}(J_a^2(X))\otimes_{\bf Z}{\bf Q}
\]
is algebraically equivalent to $0$.
Suppose that $M(X)$ has a Chow-K\"unneth decomposition
\[
M_0(X)\oplus \cdots \oplus M_{2n}(X)
\]
in $\dmeffl$.
The conjectural Bloch-Beilinson filtration on $CH^2(X)$ expects that
\begin{equation}
\label{4.4.1}
{\rm Hom}_{\dmeffl}({\bf L}^{n-2},M_r(X))
\cong
\left\{
\begin{array}{ll}
{\rm ker}\,AJ_{X,{\bf Q}}^2 & \textrm{if }r={2n-2}
\\
0 & \textrm{if }r<2n-4 \textrm{ or }r>2n-1.
\end{array}
\right.
\end{equation}

If some nonzero element in the kernel of $AJ_{X,{\bf Q}}^2$ is not algebraically equivalent to $0$, then it gives a direct summand ${\bf 1}$ of ${\rm NS}_{alg,{\bf Q}}^2(X)$, which is also a direct summand of $\underline{\rm Hom}({\bf L}^{n-2},M(X))$ in $\dmeffl$ by Theorem \ref{0.1}(1).
The induced morphism
\[
{\bf 1}\rightarrow \underline{\rm Hom}({\bf L}^{n-2},M_0(X)\oplus M_1(X)\oplus \cdots \oplus M_{2n-3}(X)\oplus M_{2n-1}(X)\oplus M_{2n}(X))
\]
in $\dmeffl$ is $0$ by \eqref{4.4.1} and the assumption that the element is in the kernel of $AJ_{X,{\bf Q}}^2$.
Thus we deduce that ${\bf 1}$ is a direct summand of $\underline{\rm Hom}({\bf L}^{n-2},M_{2n-2}(X))$.
We expect an isomorphism
\[
M_{2n-2}(X)\cong {\bf L}^{n-2}\otimes M_2(X)
\]
in $\dmeffl$.
By the cancellation theorem \cite[Theorem 16.25]{MVW}, we deduce that ${\bf 1}$ is a direct summand of
\[
\underline{\rm Hom}({\bf L}^{n-2},M_{2n-2}(X))\cong \underline{\rm Hom}({\bf L}^{n-2},{\bf L}^{n-2}\otimes M_2(X))\cong M_2(X).
\]
In particular, we have a nonzero morphism $M_2(X)\rightarrow {\bf 1}$ in $\dmeff$.
This contradicts to the conjecture \cite[Proposition 5.8]{MR1265533}.
\end{none}
\begin{none}
\label{4.6}
Let $X$ be an $n$-dimensional smooth projective variety over ${\bf C}$, and let
\[
M(X)=M_0(X)\oplus \cdots\oplus M_{2n}(X)
\]
be a conjectural Chow-K\"unneth decomposition.
Then we have
\begin{gather*}
\underline{\rm Hom}({\bf L}^{n-2},M_{2n}(X))
\simeq
{\bf L}^2,
\quad
\underline{\rm Hom}({\bf L}^{n-2},M_{2n-1}(X))
\simeq
{\rm Pic}_{\bf Q}^0(X)\otimes {\bf L},
\\
\underline{\rm Hom}({\bf L}^{n-2},M_{2n-2}(X))
\simeq
M_2(X).
\end{gather*}
We expect that the remaining components in the decomposition in Theorem \ref{0.3} can be identified as
\begin{gather*}
\underline{\rm Hom}({\bf L}^{n-2},M_{2n-3}(X))
\simeq
{\rm Griff}_{\bf Q}^2(X)\oplus J_{a,{\bf Q}}^2(X),
\\
\underline{\rm Hom}({\bf L}^{n-2},M_{2n-4}(X))
\simeq
{\rm NS}_{hom,{\bf Q}}^2(X).
\end{gather*}
We generalize this as follows.
\end{none}
\begin{conj}\label{4.2}
Let $X$ be an $n$-dimensional smooth projective variety over ${\bf C}$, and let
\[
M(X)=M_0(X)\oplus \cdots\oplus M_{2n}(X)
\]
be a conjectural Chow-K\"unneth decomposition.
Then there are isomorphisms
\begin{align*}
&\underline{\rm Hom}({\bf L}^{n-d},M_{2n-2d+1}(X))
\\
\cong\, &CH_{hom,\bf Q}^d(X)/(CH_{Jac,{\bf Q}}^d(X)+CH_{alg,{\bf Q}}^d(X)))\oplus J_{a,{\bf Q}}^d(X),
\end{align*}
\[
\underline{\rm Hom}({\bf L}^{n-d},M_{2n-2d}(X))
\cong
{\rm NS}_{hom,{\bf Q}}^d(X),
\]
for every integer $d\in [1,d]$.
\end{conj}
\bibliography{../bib}
\bibliographystyle{siam}
\end{document}